\documentclass[12pt]{article}

\usepackage{amsmath, epsfig, cite}
\usepackage{amssymb}
\usepackage{amsfonts}
\usepackage{latexsym}
\usepackage{amsthm}

\newtheorem{thm}{Theorem}[section]
\newtheorem{prop}[thm]{Proposition}

\newtheorem{cor}[thm]{Corollary}
\newtheorem{conj}[thm]{Conjecture}

\numberwithin{equation}{section}

\renewcommand{\thefootnote}{}

\begin{document}

\begin{center}
{\large\bf $q$-Supercongruences from the $q$-Saalsch\"{u}tz identity

 \footnote{ The corresponding author$^*$. Email addresses: weichuanan78@163.com (C. Wei), lydshdx@163.com (Y. Liu), xiaoxiawang@shu.edu.cn (X. Wang).}}
\end{center}

\renewcommand{\thefootnote}{$\dagger$}

\vskip 2mm \centerline{$^{1}$Chuanan Wei, $^{2}$Yudong Liu, $^{2}$Xiaoxia Wang$^*$}
\begin{center}
{$^{1}$School of Biomedical Information and Engineering,\\ Hainan Medical University, Haikou 571199, China\\
$^{2}$Department of Mathematics,\\ Shanghai University, Shanghai 200444, China}
\end{center}

%%date: January 4, 2011
%\vskip 5mm
%\noindent {\it Suggested Running title}: Two Identities of Gould

\vskip 0.7cm \noindent{\bf Abstract.} In terms of the
$q$-Saalsch\"{u}tz identity and the Chinese remainder theorem for
coprime polynomials, we establish some $q$-supercongruences modulo the
third power of a cyclotomic polynomial. In particular, we give a
$q$-analogue of a formula due to Long and Ramakrishna [Adv. Math.
290 (2016), 773--808].



\section{Introduction}
For any complex variable $x$, define the shifted-factorial by
\[(x)_{0}=1\quad \text{and}\quad (x)_{n}
=x(x+1)\cdots(x+n-1)\quad \text{when}\quad n\in\mathbb{Z}^{+}.\] Let
$p$ be an odd prime and $\mathbb{Z}_p$ denote the ring of all
$p$-adic integers. Define Morita's $p$-adic Gamma function (cf.
\cite[Chapter 7]{Robert}) to be
 \[\Gamma_{p}(0)=1\quad \text{and}\quad \Gamma_{p}(n)
=(-1)^n\prod_{\substack{1\leqslant k< n\\
p\nmid k}}k,\quad \text{when}\quad n\in\mathbb{Z}^{+}.\]
Noting $\mathbb{N}$ is a dense subset of $\mathbb{Z}_p$ associated with the $p$-adic norm $|\cdot|_p$, for each
$x\in\mathbb{Z}_p$, the definition of $p$-adic Gamma function can be extended as
 \[\Gamma_{p}(x)
=\lim_{\substack{n\in\mathbb{N}\\
|x-n|_p\to0}}\Gamma_{p}(n).\]
 Two important properties of the $p$-adic Gamma function can be expressed as follows:
  \begin{equation*}
\frac{\Gamma_p(x+1)}{\Gamma_p(x)}=
\begin{cases}
 -x, &\text{if $p\nmid x$,}\\[10pt]
-1, &\text{if $p\,|\,x$,}
 \end{cases}
 \end{equation*}
  \begin{equation*}
\Gamma_p(x)\Gamma_p(1-x)=(-1)^{\langle-x\rangle_p-1},
 \end{equation*}
where $\langle x\rangle_p$ indicates the least nonnegative residue of $x$ modulo $p$, i.e., $\langle x\rangle_p\equiv x\pmod p$
and $\langle x\rangle_p\in\{0,1,\ldots,p-1\}$.
  In 2006, Long and Ramakrishna
\cite[Proposition 25]{LR} proved that, for any prime $p$,
\begin{equation}\label{eq:long}
\sum_{k=0}^{p-1}\frac{(1/3)_k^3}{k!^3}\equiv
\begin{cases} \displaystyle \Gamma_p(1/3)^6  \pmod{p^3}, &\text{if $p\equiv 1\pmod 6$,}\\[10pt]
\displaystyle -\frac{p^2}{3}\Gamma_p(1/3)^6\pmod{p^3}, &\text{if
$p\equiv 5\pmod 6$.}
\end{cases}
\end{equation}

 For any complex numbers $x$ and $q$, define the $q$-shifted factorial
 by
 \begin{equation*}
(x;q)_{0}=1\quad\text{and}\quad
(x;q)_n=(1-x)(1-xq)\cdots(1-xq^{n-1})\quad \text{when}\quad
n\in\mathbb{Z}^{+}.
 \end{equation*}
For simplicity, we also adopt the compact notation
\begin{equation*}
(x_1,x_2,\dots,x_m;q)_{n}=(x_1;q)_{n}(x_2;q)_{n}\cdots(x_m;q)_{n}.
 \end{equation*}
Following Gasper and Rahman \cite{Gasper}, define the basic
hypergeometric series $_{r+1}\phi_{r}$ to be
$$
_{r+1}\phi_{r}\left[\begin{array}{c}
a_1,a_2,\ldots,a_{r+1}\\
b_1,b_2,\ldots,b_{r}
\end{array};q,\, z
\right] =\sum_{k=0}^{\infty}\frac{(a_1,a_2,\ldots, a_{r+1};q)_k}
{(q,b_1,b_2,\ldots,b_{r};q)_k}z^k.
$$
Then the $q$-Saalsch\"{u}tz identity (cf. \cite[Appendix
(II.12)]{Gasper}) can be stated as
\begin{align}
& _{3}\phi_{2}\!\left[\begin{array}{cccccccc}
 q^{-n}, a, b\\
  c,  q^{1-n}ab/c
\end{array};q,\, q \right]
=\frac{(c/a, c/b;q)_{n}}
 {(c, c/ab;q)_{n}}. \label{saal}
\end{align}

Recently, Guo \cite[Theorem 1.1]{Guo-fac} found that, for positive integers $d,n,r$ such that $d\geq2$, $n\equiv-r\pmod d$, $n\geq d-r$,
 $r\leq d-2$, and $\gcd(d,r)=1$,
\begin{align}
&\sum_{k=0}^{n-1}\frac{(q^r;q^d)_k^d}{(q^d;q^d)_k^d}q^{dk}
 \equiv0\pmod{\Phi_n(q)^2}.
\label{eq:guo-a}
\end{align}
Here and throughout the paper, $\Phi_n(q)$ stands for
the $n$-th cyclotomic polynomial in $q$:
\begin{equation*}
\Phi_n(q)=\prod_{\substack{1\leqslant k\leqslant n\\
\gcd(k,n)=1}}(q-\zeta^k),
\end{equation*}
where $\zeta$ is an $n$-th primitive root of unity. For more
$q$-analogues of supercongruences, we refer the reader to
\cite{Guo-adb,Guo-rima,Guo-ijnt,GS20,GS20c,GS,GuoZu,LW,LP,Tauraso,WY-a,Zu19}.

Motivated by the work just mentioned, we shall establish the
following two theorems.

\begin{thm}\label{thm-a}
Let $d\geq 2$ and $n$ be positive integers with $n\equiv
1\pmod d$. Then, modulo $\Phi_n(q)^3$,
\begin{align*}
\sum_{k=0}^{(n-1)/d}\frac{(q;q^d)_k^3}{(q^3;q^d)_k(q^d;q^d)_k^2}q^{dk}
&\equiv
q^{(n-1)/d}\frac{(q^2,q^{d-1};q^d)_{(n-1)/d}}{(q^3,q^d;q^d)_{(n-1)/d}}
\\[5pt]
&\times\bigg\{1+[n]^2\sum_{i=1}^{(n-1)/d}\frac{q^{di-d+2}}{[di-d+2]^2}\bigg\},
\end{align*}
where  $[n]=(1-q^n)/(1-q)$ is the $q$-integer.
\end{thm}

\begin{thm}\label{thm-b}
Let $n$ be a positive integer with $n\equiv
2\pmod 3$. Then, modulo $\Phi_n(q)^3$,
\begin{align*}
\sum_{k=0}^{(2n-1)/3}\frac{(q;q^3)_k^3}{(q^3;q^3)_k^3}q^{3k} &\equiv
q^{(2n-1)/3}\frac{(q^2;q^3)_{(2n-1)/3}^2}{(q^3;q^3)_{(2n-1)/3}^2}
\bigg\{1-[2n]^2\sum_{i=1}^{(2n-1)/3}\frac{q^{3i-1}}{[3i-1]^2}\bigg\}.
\end{align*}
\end{thm}

Although we don't discover the general form of Theorem \ref{thm-b} including the parameter $d$, it is not difficult to understand that Theorem \ref{thm-a} with $d=3$
and Theorem \ref{thm-b} give a $q$-analogue of \eqref{eq:long}.

Letting $n=p$ be an odd prime and taking $q\to 1$ in the above two
theorems, we obtain the following conclusion.
\begin{cor}\label{cor-b}
Let $p$ be an odd prime such that $p\equiv t\pmod{3}$ with $t\in\{1,2\}$. Then
\begin{align}\label{eq:wei-a}
\sum_{k=0}^{(tp-1)/3}\frac{(1/3)_k^3}{k!^3} \equiv
\frac{(2/3)_{(tp-1)/3}^2}{(1)_{(tp-1)/3}^2}\bigg\{1+(-1)^{t-1}(tp)^2\sum_{i=1}^{(tp-1)/3}\frac{1}{(3i-1)^2}\bigg\}\pmod{p^3}.
\end{align}
\end{cor}

In order to explain the equivalence of \eqref{eq:long}
 and \eqref{eq:wei-a}, we need to verify the
following relations.

\begin{prop}\label{prop-a}
Let $p$ be an odd prime. Then
\begin{align}
&\frac{(2/3)_{(p-1)/3}^2}{(1)_{(p-1)/3}^2}\bigg\{1+p^2\sum_{i=1}^{(p-1)/3}\frac{1}{(3i-1)^2}\bigg\}\equiv
\Gamma_p(1/3)^6\pmod{p^3} \label{eq:wei-oa}
\end{align}
if $p\equiv 1\pmod 3$, and
\begin{align}
&\frac{(2/3)_{(2p-1)/3}^2}{(1)_{(2p-1)/3}^2}\bigg\{1-4p^2\sum_{i=1}^{(2p-1)/3}\frac{1}{(3i-1)^2}\bigg\}\equiv
-\frac{p^2}{3}\Gamma_p(1/3)^6\pmod{p^3}
 \label{eq:wei-ob}
\end{align}
if $p\equiv 2\pmod 3$.
\end{prop}

The rest of the paper is arranged as follows. By means of
\eqref{saal} and the Chinese remainder theorem for coprime
polynomials, a $q$-supercongruence modulo
$(1-aq^{tn})(a-q^{tn})(b-q^{tn})$ will be derived in Section 2. Then
it is used to provide proofs of Theorems \ref{thm-a} and
\ref{thm-b} in the same section. Finally, the proof of Proposition
\ref{prop-a} will be displayed in Section 3.

%%%%%%%%%%%%%%%%%%%%%%%%%%%%%%%%%%%%%%%%%%%%%%%%%%%%%%%%%%%%%%%%%%%%%%%%%%%%%%%%%%%%%%%%%%%%%%%%%%%%%%%%%%%%%%%%%%%%%%%%%%%%%%%%%%%%%%%%%%%%%%%%%%%%%%%%%%%%%%%%%%%%%%%%%%%%%%%%%%%%%%%%%%%%%%%%%%%%%%%%%%%%%%%%%%%
\section{Proofs of Theorems \ref{thm-a} and \ref{thm-b}}
%%%%%%%%%%%%%%%%%%%%%%%%%%%%%%%%%%%%%%%%%%%%%%%%%%%%%%%%%%%%%%%%%%%%%%%%%%%%%%%%%%%%%%%%%%%%%%%%%%%%%%%%%%%%%%%%%

In order to prove Theorems \ref{thm-a} and \ref{thm-b}, we require the following
united parameter extension of them.

\begin{thm}\label{thm-c}
Let $d\geq 2,n$ be positive integers,  $t\in\{1,d-1\}$ and $n\equiv
t\pmod d$. Then, modulo $(1-aq^{tn})(a-q^{tn})(b-q^{tn})$,
\begin{align}
&\sum_{k=0}^{(tn-1)/d}\frac{(aq,q/a,q/b;q^d)_k}{(q^d,c,q^{d+3}/bc;q^d)_k}q^{dk}
\notag\\[5pt]
&\quad\equiv\,\frac{(1-aq^{tn})(a-q^{tn})}{(a-b)(1-ab)}\frac{(c/aq,ac/q;q^d)_{(tn-1)/d}}{(c,c/q^2;q^d)_{(tn-1)/d}}
\notag\\[5pt]
&\quad+\:\frac{(b-q^{tn})(ab-1-a^2+aq^{tn})}{(a-b)(1-ab)}\frac{(q/b)^{(tn-1)/d}(bc/q,q^{d+2}/c;q^d)_{(tn-1)/d}}{(c,q^{d+3}/bc;q^d)_{(tn-1)/d}}.
\label{eq:wei-aa}
\end{align}
\end{thm}

\begin{proof}
When $a=q^{-tn}$ or $a=q^{tn}$, the left-hand side of
\eqref{eq:wei-aa} is equal to
\begin{align}
\sum_{k=0}^{(tn-1)/d}\frac{(q^{1-tn},q^{1+tn},q/b;q^d)_k}{(q^d,c,q^{d+3}/bc;q^d)_k}q^{dk}
= {_{3}\phi_{2}}\!\left[\begin{array}{cccccccc}
 q^{1-tn},  q^{1+tn}, q/b\\
 c,q^{d+3}/bc
\end{array};q^d,\, q^d \right].
 \label{eq:saal-aa}
\end{align}
Via \eqref{saal}, the right-hand side of \eqref{eq:saal-aa} can be
written as
\begin{align*}
(q/b)^{(tn-1)/d}\frac{(bc/q,q^{d+2}/c;q^d)_{(tn-1)/d}}{(c,q^{d+3}/bc;q^d)_{(tn-1)/d}}.
\end{align*}
Since $(1-aq^{tn})$ and $(a-q^{tn})$ are relatively prime
polynomials, we get the following result: modulo
$(1-aq^{tn})(a-q^{tn})$,
\begin{align}
\sum_{k=0}^{(tn-1)/d}\frac{(aq,q/a,q/b;q^d)_k}{(q^d,c,q^{d+3}/bc;q^d)_k}q^{dk}\equiv
 (q/b)^{(tn-1)/d}\frac{(bc/q,q^{d+2}/c;q^d)_{(tn-1)/d}}{(c,q^{d+3}/bc;q^d)_{(tn-1)/d}}. \label{eq:wei-bb}
\end{align}

When $b=q^{tn}$, the left-hand side of  \eqref{eq:wei-aa} is equal
to
\begin{align}
\sum_{k=0}^{(tn-1)/d}\frac{(aq,q/a,q^{1-tn};q^d)_k}{(q^d,c,q^{d+3-tn}/c;q^d)_k}q^{dk}
= {_{3}\phi_{2}}\!\left[\begin{array}{cccccccc}
 aq,q/a,q^{1-tn}\\
 c,q^{d+3-tn}/c
\end{array};q^d,\, q^d \right].
 \label{eq:saal-bb}
\end{align}
Through \eqref{saal}, the right-hand side of \eqref{eq:saal-bb} can be
evaluated as
\begin{align*}
\frac{(c/aq,ac/q;q^d)_{(tn-1)/d}}{(c,c/q^2;q^d)_{(tn-1)/d}}.
\end{align*}
Therefore, modulo $(b-q^{tn})$,
\begin{align}
\sum_{k=0}^{(tn-1)/d}\frac{(aq,q/a,q/b;q^d)_k}{(q^d,c,q^{d+3}/bc;q^d)_k}q^{dk}\equiv
 \frac{(c/aq,ac/q;q^d)_{(tn-1)/d}}{(c,c/q^2;q^d)_{(tn-1)/d}}.\label{eq:wei-cc}
\end{align}

It is clear that the polynomials $(1-aq^{tn})(a-q^{tn})$ and
$(b-q^{tn})$ are relatively prime. Noting the $q$-congruences
\begin{align*}
&\frac{(b-q^{tn})(ab-1-a^2+aq^{tn})}{(a-b)(1-ab)}\equiv1\pmod{(1-aq^{tn})(a-q^{tn})},
\\[5pt]
&\qquad\qquad\frac{(1-aq^{tn})(a-q^{tn})}{(a-b)(1-ab)}\equiv1\pmod{(b-q^{tn})}
\end{align*}
and employing the Chinese remainder theorem for coprime polynomials,
we arrive at Theorem \ref{thm-c} from \eqref{eq:wei-bb} and
\eqref{eq:wei-cc}.
\end{proof}

\begin{proof}[Proof of Theorem \ref{thm-a}]
Letting $b\to1, c\to q^3, t=1$ in Theorem \ref{thm-c}, we obtain the
formula: modulo $\Phi_n(q)(1-aq^{n})(a-q^{n})$,
\begin{align*}
&\sum_{k=0}^{(n-1)/d}\frac{(aq,q/a,q;q^d)_k}{(q^d,q^d,q^{3};q^d)_k}q^{dk}
\notag\\[5pt]
&\quad\equiv\,q^{(n-1)/d}\frac{(q^2,q^{d-1};q^d)_{(n-1)/d}}{(q^3,q^{d};q^d)_{(n-1)/d}}+\frac{(1-aq^{n})(a-q^{n})}{(1-a)^2}
\notag\\[5pt]
&\quad\quad\times\:\bigg\{q^{(n-1)/d}\frac{(q^2,q^{d-1};q^d)_{(n-1)/d}}{(q^3,q^{d};q^d)_{(n-1)/d}}-\frac{(aq^2,q^{2}/a;q^d)_{(n-1)/d}}{(q,q^{3};q^d)_{(n-1)/d}}\bigg\}.
\end{align*}
In term of the relation:
\begin{align}\label{relation-a}
&q^{(n-1)/d}\frac{(q^{d-1};q^d)_{(n-1)/d}}{(q^{d};q^d)_{(n-1)/d}}=\frac{(q^{2-n};q^d)_{(n-1)/d}}{(q^{1-n};q^d)_{(n-1)/d}}
\equiv\frac{(q^{2};q^d)_{(n-1)/d}}{(q;q^d)_{(n-1)/d}}\pmod{\Phi_n(q)},
\end{align}
we get the conclusion: modulo
$\Phi_n(q)(1-aq^{n})(a-q^{n})$,
\begin{align}
&\sum_{k=0}^{(n-1)/d}\frac{(aq,q/a,q;q^d)_k}{(q^d,q^d,q^{3};q^d)_k}q^{dk}
\notag\\[5pt]
&\quad\equiv\,q^{(n-1)/d}\frac{(q^2,q^{d-1};q^d)_{(n-1)/d}}{(q^3,q^{d};q^d)_{(n-1)/d}}+\frac{(1-aq^{n})(a-q^{n})}{(1-a)^2}
\notag\\[5pt]
&\quad\quad\times\:\bigg\{\frac{(q^2,q^{2};q^d)_{(n-1)/d}}{(q,q^{3};q^d)_{(n-1)/d}}-\frac{(aq^2,q^{2}/a;q^d)_{(n-1)/d}}{(q,q^{3};q^d)_{(n-1)/d}}\bigg\}.
\label{eq:wei-dd}
\end{align}
 By the L'H\^{o}spital rule, we have
\begin{align*}
&\lim_{a\to1}\frac{(1-aq^{n})(a-q^{n})}{(1-a)^2}\bigg\{\frac{(q^2,q^{2};q^d)_{(n-1)/d}}{(q,q^{3};q^d)_{(n-1)/d}}-\frac{(aq^2,q^{2}/a;q^d)_{(n-1)/d}}{(q,q^{3};q^d)_{(n-1)/d}}\bigg\}\\[5pt]
&\quad=[n]^2\frac{(q^2;q^d)_{(n-1)/d}^2}{(q,q^3;q^d)_{(n-1)/d}}\sum_{i=1}^{(n-1)/d}\frac{q^{di-d+2}}{[di-d+2]^2}.
\end{align*}
Letting $a\to1$ in \eqref{eq:wei-dd} and utilizing the above limit,
we get the $q$-supercongruence: modulo $\Phi_n(q)^3$,
\begin{align}
\sum_{k=0}^{(n-1)/d}\frac{(q;q^d)_k^3}{(q^3;q^d)_k(q^d;q^d)_k^2}q^{dk}
&\equiv
q^{(n-1)/d}\frac{(q^2,q^{d-1};q^d)_{(n-1)/d}}{(q^3,q^d;q^d)_{(n-1)/d}}
\notag\\[5pt]
&\quad+[n]^2\frac{(q^2;q^d)_{(n-1)/d}^2}{(q,q^3;q^d)_{(n-1)/d}}\sum_{i=1}^{(n-1)/d}\frac{q^{di-d+2}}{[di-d+2]^2}.
\label{relation-b}
\end{align}
Substituting \eqref{relation-a} into \eqref{relation-b}, we are led to  Theorem
\ref{thm-a}.
\end{proof}

\begin{proof}[Proof of Theorem \ref{thm-b}]
Letting $b\to1, c\to q^3, d=3, t=2$ in Theorem \ref{thm-c}, we arrive at the result: modulo $\Phi_n(q)(1-aq^{2n})(a-q^{2n})$,
\begin{align}
&\sum_{k=0}^{(2n-1)/3}\frac{(aq,q/a,q;q^3)_k}{(q^{3};q^3)_k^3}q^{3k}
\notag\\[5pt]
&\quad\equiv\,q^{(2n-1)/3}\frac{(q^2;q^3)_{(2n-1)/3}^2}{(q^3;q^3)_{(2n-1)/3}^2}+\frac{(1-aq^{2n})(a-q^{2n})}{(1-a)^2}
\notag\\[5pt]
&\quad\quad\times\:\bigg\{q^{(2n-1)/3}\frac{(q^2;q^3)_{(2n-1)/3}^2}{(q^3;q^3)_{(2n-1)/3}^2}-\frac{(aq^2,q^{2}/a;q^3)_{(2n-1)/3}}{(q,q^{3};q^3)_{(2n-1)/3}}\bigg\}.
\label{relation-c}
\end{align}
Similar to the proof of \eqref{relation-b}, we can deduce from \eqref{relation-c} that,
modulo $\Phi_n(q)^3$,
\begin{align}
\sum_{k=0}^{(2n-1)/3}\frac{(q;q^3)_k^3}{(q^3;q^3)_k^3}q^{3k}
&\equiv q^{(2n-1)/3}\frac{(q^2;q^3)_{(2n-1)/3}^2}{(q^3;q^3)_{(2n-1)/3}^2}
\notag\\[5pt]
&\quad+[2n]^2\frac{(q^2;q^3)_{(2n-1)/3}^2}{(q,q^3;q^3)_{(2n-1)/3}}\sum_{i=1}^{(2n-1)/3}\frac{q^{3i-1}}{[3i-1]^2}.
\label{eq:wei-ee}
\end{align}
Moreover, it is routine to verify the congruence:
\begin{align}
(q;q^3)_{(2n-1)/3}&=(1-q)(1-q^4)\cdots(1-q^{2n-3})
\notag\\[5pt]
&\equiv(1-q^{1-2n})(1-q^{4-2n})\cdots(1-q^{-3})
\notag\\[5pt]
&=(-1)^{(2n-1)/3}q^{-(n+1)(2n-1)/3}(q^3;q^3)_{(2n-1)/3}
\notag\\[5pt]
&\equiv-q^{-(2n-1)/3}(q^3;q^3)_{(2n-1)/3}\pmod{\Phi_n(q)}.
\label{eq:wei-ff}
\end{align}
 The combination of
\eqref{eq:wei-ee} and \eqref{eq:wei-ff} produces Theorem
\ref{thm-b}.
\end{proof}
%%%%%%%%%%%%%%%%%%%%%%%%%%%%%%%%%%%%%%%%%%%%%%%%%%%%%%%%%%%%%%%%%%%%%%%%%%%%%%%%%%%%%%%%%%%%%%%%%%%%%%%%%%%%%%%%%%%%%%%%%%%%%%%%%%%%%%%%%%%%%%%%%%%%%%%%%%%%%%%%%%%%%%%%%%%%%%%%%%%%%%%%%%%%%%%%%%%%%%%%%%%%%%%
\section{Proof of Proposition \ref{prop-a}}

 Now we begin to prove the supercongruence \eqref{eq:wei-oa}.
Let $\Gamma_p^{'}(x)$ and $\Gamma_p^{''}(x)$ respectively be the
first derivative and second derivative of $\Gamma_p(x)$. By means of
the properties of the $p$-adic Gamma function, we obtain
\begin{align}
\frac{(2/3)_{(p-1)/3}^2}{(1)_{(p-1)/3}^2}&=\bigg\{\frac{\Gamma_p((1+p)/3)\Gamma_p(1)}{\Gamma_p(2/3)\Gamma_p((2+p)/3)}\bigg\}^2
\notag\\[5pt]
&=\big\{\Gamma_p(1/3)\Gamma_p((1+p)/3)\Gamma_p((1-p)/3)\big\}^2
\notag\\[5pt]
&\equiv\Gamma_p(1/3)^2\bigg\{\Gamma_p(1/3)+\Gamma_p^{'}(1/3)\frac{p}{3}+\Gamma_p^{''}(1/3)\frac{p^2}{18}\bigg\}^2
\notag\\[5pt]
&\quad\times\bigg\{\Gamma_p(1/3)-\Gamma_p^{'}(1/3)\frac{p}{3}+\Gamma_p^{''}(1/3)\frac{p^2}{18}\bigg\}^2
\notag\\[5pt]
&\equiv\Gamma_p(1/3)^6\bigg\{1-\frac{2p^2}{9}G_1(1/3)^2+\frac{2p^2}{9}G_2(1/3)\bigg\}\pmod{p^3},
\label{eq:wei-aaa}
\end{align}
where $G_1(x)=\Gamma_p^{'}(x)/\Gamma_p(x)$ and
$G_2(x)=\Gamma_p^{''}(x)/\Gamma_p(x)$.

Let
\[H_{m}
  =\sum_{k=1}^m\frac{1}{k},\quad H_{m}^{(2)}
  =\sum_{k=1}^m\frac{1}{k^{2}},\]
\begin{align*}
\quad
{H}_m^{(2)}(p)=\sum_{\substack{1\leq k\leq m\\
p\nmid k}}\frac{1}{k^{2}},
\quad
\mathfrak{H}_m^{(2)}(p)=\sum_{\substack{1\leq k_1<k_2\leq m\\
p\nmid k_1k_2}}\frac{1}{k_1k_2}.
\end{align*}
 Via the following two relations from Pan, Tauraso and Wang\cite[Theorem 4.1]{PTW}:
\begin{align*}
&G_1(1/3)\equiv G_1(0)+{H}_{(2p-2)/3}\pmod{p},
 \\[5pt]
&G_2(1/3)\equiv
G_2(0)+2G_1(0){H}_{(2p-2)/3}+2\mathfrak{H}_{(2p^2-2)/3}^{(2)}(p)\pmod{p^2}
\end{align*}
 and the equation (cf. \cite[Lemma 4.3]{Wang}):
\begin{align*}
G_2(0)=G_1(0)^2,
\end{align*}
we get
\begin{align}
G_2(1/3)-G_1(1/3)^2&\equiv2\mathfrak{H}_{(2p^2-2)/3}^{(2)}(p)-{H}_{(2p-2)/3}^2
\notag\\[5pt]
&=-{H}_{(2p^2-2)/3}^{(2)}(p)
\notag
\end{align}
\begin{align}
&\equiv-\sum_{i=1}^{(2p-2)/3}\frac{1}{(i+\frac{2p^2-2p}{3})^2}
\notag\\[5pt]
&\equiv-{H}_{(2p-2)/3}^{(2)}\pmod{p}. \label{eq:wei-bbb}
\end{align}
In view of \eqref{eq:wei-aaa} and \eqref{eq:wei-bbb}, the left-hand side of \eqref{eq:wei-oa} is congruent to
\begin{align}
 &\Gamma_p(1/3)^6\bigg\{1-\frac{2p^2}{9}{H}_{(2p-2)/3}^{(2)}\bigg\}\bigg\{1+p^2\sum_{i=1}^{(p-1)/3}\frac{1}{(3i-1)^2}\bigg\}
\notag\\[5pt]
&\quad\equiv
\Gamma_p(1/3)^6\bigg\{1-\frac{2p^2}{9}{H}_{(2p-2)/3}^{(2)}+p^2\sum_{i=1}^{(p-1)/3}\frac{1}{(3i-1)^2}\bigg\}\pmod{p^3}.
 \label{eq:wei-ccc}
\end{align}
It is easy to see that
 \begin{align}
\sum_{i=1}^{(p-1)/3}\frac{1}{(3i-1)^2}&=H_{p-1}^{(2)}-\sum_{i=1}^{(p-1)/3}\frac{1}{(3i-2)^2}-\frac{1}{9}H_{(p-1)/3}^{(2)}
 \notag\\[5pt]
 &\equiv
-\sum_{i=1}^{(p-1)/3}\frac{1}{(3i-2)^2}-\frac{1}{9}H_{(p-1)/3}^{(2)}
\notag\\[5pt]
&=-\sum_{i=1}^{(p-1)/3}\frac{1}{(p-3i)^2}-\frac{1}{9}H_{(p-1)/3}^{(2)}
 \notag\\[5pt]
 &\equiv
-\frac{2}{9}H_{(p-1)/3}^{(2)}
=-\frac{2}{9}\sum_{i=(2p+1)/3}^{p-1}\frac{1}{(p-i)^2}
\notag\\[5pt]
 &\equiv
-\frac{2}{9}\sum_{i=(2p+1)/3}^{p-1}\frac{1}{i^2}
\equiv\frac{2}{9}H_{(2p-2)/3}^{(2)}
 \pmod{p}. \label{eq:wei-ddd}
\end{align}
 Substituting \eqref{eq:wei-ddd} into  \eqref{eq:wei-ccc}, we confirm the correctness of \eqref{eq:wei-oa}. The proof of \eqref{eq:wei-ob}
 admits a similar process. The corresponding details are omitted here.

In 2015, Swisher\cite[(H.3)]{Swisher} conjectured a nice supercongruence: for any prime $p\equiv 1 \pmod{4}$,
\begin{equation*}
\sum_{k=0}^{(p^r-1)/2}\frac{(1/2)_k^3}{k!^3}
\equiv -\Gamma_p(1/4)^4\sum_{k=0}^{(p^{r-1}-1)/2}\frac{(1/2)_k^3}{k!^3}
\pmod {p^{3r}},
\end{equation*}
where $r$ is a positive integer. On the basis of numerical calculations, we would like to put forward
the following conjecture.

\begin{conj}Let $p$ be a prime with $p\equiv 1 \pmod{3}$ and $r$ a positive integer. Then
\begin{align*}
&\sum_{k=0}^{(p^r-1)/3}\frac{(1/3)_k^3}{k!^3}
\equiv \Gamma_p(1/3)^6\sum_{k=0}^{(p^{r-1}-1)/3}\frac{(1/3)_k^3}{k!^3}
\pmod {p^{3r}}.\\
\end{align*}
\end{conj}

{\bf{Acknowledgments}}\\

The work is supported by the National Natural Science Foundations of China (Nos. 12071103 and
11661032).

%%%%%%%%%%%%%%%%%%%%%%%%%%%%%%%%%%%%%%%%%%%%%%%%%%%%%%%%%%%%%%%%%%%%%%%%%%%%%%%%%%%%%%%%%%%%%%%%%%%%%%%%%%%%%%%%%%%%%%%%%%%%%%%%%%%%%%%%%%%%%%%%%%%%%%%%%%%%%%%%%%%%%%%%%%%%%%%%%%%%%%%%%%%%%%%%%%%%%%%%%%


\begin{thebibliography}{99}
\small \setlength{\itemsep}{-.8mm}

\bibitem{Gasper}G. Gasper, M. Rahman, Basic Hypergeometric Series (2nd edition),
Cambridge University Press, Cambridge, 2004.

\bibitem{Guo-fac}V.J.W. Guo,
Factors of some truncated basic hypergeometric hypergeometric
series, J. Math. Anal. Appl. 476 (2019), 851--859.

\bibitem{Guo-adb}V.J.W. Guo, $q$-Supercongruences modulo the fourth power of a cyclotomic polynomial via creative microscoping, Adv. Appl. Math. 120 (2020), Art. 102078.

\bibitem{Guo-rima}V.J.W. Guo,
Proof of some $q$-supercongruences modulo the fourth power of a  cyclotomic polynomial,
Results Math. 75 (2020), Art. 77.

\bibitem{Guo-ijnt}V.J.W. Guo, A further $q$-analogue of Van Hamme's (H.2) supercongruence for primes $p\equiv 3\pmod{4}$, Int. J. Number Theory, to appear.

\bibitem{GS20} V.J.W. Guo and M.J. Schlosser, A new family of $q$-supercongruences modulo the fourth power of a cyclotomic polynomial,
Results Math. 75 (2020), Art. 155.

\bibitem{GS20c}
V.J.W. Guo, M.J. Schlosser, A family of $q$-hypergeometric
congruences modulo the fourth power of a cyclotomic polynomial,
Israel J. Math., to appear.

\bibitem{GS}V.J.W. Guo, M.J. Schlosser, Some $q$-supercongruences from transformation formulas for basic
hypergeometric series, Constr. Approx., to appear.

\bibitem{GuoZu}V.J.W. Guo, W. Zudilin, A $q$-microscope for supercongruences, Adv. Math. 346 (2019), 329--358.

\bibitem{LW}L. Li, S.-D. Wang, Proof of a $q$-supercongruence conjectured by Guo and Schlosser,
Rev. R. Acad. Cienc. Exactas Fs. Nat., Ser. A Mat. RACSAM 114 (2020), Art. 190.

\bibitem{LP}J.-C. Liu, F. Petrov, Congruences on sums of $q$-binomial coefficients, Adv. Appl. Math. 116 (2020), Art.~102003.

\bibitem{LR} L. Long, R. Ramakrishna, Some supercongruences occurring in truncated hypergeometric series, Adv. Math. 290 (2016), 773--808.

\bibitem{PTW}H. Pan, R. Tauraso, C. Wang, A local-global theorem for $p$-adic supergruences, preprint, 2020,
arXiv:1909.08183v1.

\bibitem{Robert} A.M. Robert, A Course in $p$-Adic Analysis, Graduate Texts in Mathematics, SpringerVerlag, New York, 2000.

\bibitem{Swisher}H. Swisher, On the supercongruence conjectures of Van Hamme, Res. Math. Sci. 2 (2015), Art. 18.

\bibitem{Tauraso}R. Tauraso, $q$-Analogs of some congruences involving Catalan numbers, Adv. Appl. Math. 48 (2009), 603--614.

\bibitem{Wang}C. Wang, H. Pan, Supercongruences concerning truncated hypergeometric series, preprint, 2018,
arXiv:1806.02735v2.

\bibitem{WY-a}X. Wang, M. Yue, Some $q$-supercongruences from Watson's
$_8\phi_7$ transformation formula, Results Math. 75 (2020), Art.~71.

\bibitem{Zu19}W. Zudilin, Congruences for $q$-binomial coefficients, Ann. Combin. 23 (2019), 1123--1135.

\end{thebibliography}
\end{document}